\documentclass{amsart}
\usepackage{amssymb}
\usepackage{amsbsy, amsthm, amsmath,amstext, amsopn}
\usepackage{tikz-cd}
\usepackage{multicol}
\usepackage{tikz}
\usepackage{mathtools}
\usepackage{bbm}
\usetikzlibrary{arrows,snakes,backgrounds}
\usetikzlibrary{decorations.markings}
\usetikzlibrary{matrix}
\usetikzlibrary{graphs}

\theoremstyle{definition}
\newtheorem{theorem}{Theorem}[section]
\newtheorem{prop}[theorem]{Proposition}

\theoremstyle{definition}

\newtheorem{fact}[theorem]{Fact}

\theoremstyle{remark}

\newcommand{\Z}{\mathbb{Z}}

\newcommand{\A}{\mathcal{A}}
\newcommand{\B}{\mathcal{B}}
\newcommand{\X}{\mathcal{X}}

\newcommand{\cO}{\mathcal{O}}

\newcommand{\Spec}{\operatorname{Spec}}
\newcommand{\Conf}{\operatorname{Conf}}

\newcommand\ontop[2]{\genfrac{}{}{0pt}{}{#1}{#2}}

\newcounter{Qcount}
\begin{document}

\stepcounter{Qcount}

\title{Generalized Minors and Tensor Invariants}

\author{Ian Le}
\address{Perimeter Institute for Theoretical Physics\\ Waterloo, ON N2L 2Y5}
\email{ile@perimeterinstitute.ca}

\author{Sammy Luo}
\address{Stanford University\\ Stanford, CA 94305}
\email{sammyluo@stanford.edu}

\begin{abstract} In the paper \cite{BFZ}, the authors define functions on double Bruhat cells which they call generalized minors. By relating certain double Bruhat cells to the spaces $\Conf_3 \A_G$ and $\Conf_4 \A_G$, we give formulas for these generalized minors as tensor invariants. This allows us to verify certain weight identities conjectured in \cite{Le2}. In that paper, we showed that the weights of the tensor invariants determined the quiver for the cluster structure on $\Conf_3 \A_G$. In this paper, we also show the reverse--that the weights of tensor invariants can by computed from the structure of the quiver. The weight identities are important because they are necessary for the existence of cluster structures on the moduli space of framed local systems $\X_{G',S}$ and $\A_{G,S}$.
\end{abstract}

\maketitle

\tableofcontents

\section{Introduction}

Let $G$ be a group which is simply connected, split and semi-simple, and let $U$ be a maximal unipotent subgroup. We can consider the principal flag variety $G/U$, which we denote $\A_G$. We can then consider the configuration space of three principal flags $\Conf_3 \A_G$, which is the quotient of $\A_G^3$ by the diagonal left action of $G$.

The space $\Conf_3 \A_G$ has a cluster structure which plays a fundamental role in higher Teichmuller theory \cite{FG1}. A cIuster structure is roughly two pieces of data: a quiver (a directed graph without cycles), as well as a set functions called cluster variables, with one cluster variable associated with each vertex of the quiver.

In order to construct the cluster structure on $\Conf_3 \A_G$, it is useful to describe cluster variables on $\Conf_3 \A_G$ in terms of both generalized minors and as tensor invariants. In this paper we present a simple formula that expresses a generalized minor in terms of a tensor invariant. We will not prove this formula in full generality--we leave that task to the more comprehensive upcoming work of Goncharov and Shen, where they give a nice conceptual proof that fits into a broader framework. Let us remark that one can verify the formula in any particular case by direct calculation, as was done in \cite{Le} for the classical groups.

Instead, we use the formula to verify the conjectures outlined in the paper \cite{Le2}. There, we showed that if these conjectures hold, then the cluster variables determine the quiver. As a consequence, we get a obtain an explicit construction of this cluster structure--we get a simple description of the quiver as well as formulas for the tensor invariant space the cluster variables lie in. Note that once the cluster variables are given as tensor invariants, the quiver for the cluster algebra is actually over-determined.

Next, we explain a sort of converse: we show that the quiver for $\Conf_3 \A_G$ determines the cluster variables as tensor invariants. The structure of the quiver gives us a recursive procedure for calculating which tensor invariant space each cluster variable must lie in. We show that this answer coincides with the formula for the generalized minor as a tensor invariant, giving a self-consistency check for our formulas.

We now step back to give some context. Let $S$ be a topological surface $S$ with non-empty boundary,  and let $G'$ be the adjoint form of this group. One can then consider the moduli space of $G$-local systems on the surface $S$. We can consider two variations of this space, $\X_{G',S}$ and $\A_{G,S}$ \cite{FG1}. These spaces are moduli of local systems with the additional data of framing at the boundary of $S$.

The pair of spaces $(\X_{G',S}, \A_{G,S})$ carries an important structure: they form a {\it cluster ensemble}. This means that
\begin{itemize}
\item $\X_{G',S}$ has the structure of a cluster $\X$-variety; 
\item $\A_{G,S}$ has the structure of a cluster $\A$-variety; 
\item There is a map $p: \A_{G,S} \rightarrow \X_{G',S}$ that intertwines these structures.
\end{itemize}
Detailed definitions can be found in \cite{FG2}. By using the formalism of cluster ensembles as well as cutting and gluing arguments, one can reduce the construction of the cluster ensemble structure on the pair $(\X_{G',S}, \A_{G,S})$ to a simpler problem: to understand the cluster structure on $\A_{G,S}$ in the case that $S$ is a disc with three marked points on the boundary.

For these reasons, in this paper we concentrate on the case where $S$ is a disk with three marked points on the boundary. In this case, the moduli of (framed) local systems is given by $\Conf_3 \A_G$. It turns out that it will also be interesting and useful to study the case where $S$ is a disc with four marked points on the boundary, where the space of framed local systems is $\Conf_4 \A_G$.

Let us explain the organization of this paper. In Section 2, we define the spaces $\Conf_n \A_G$ and give some background on cluster algebras. In Section 3, we start by spelling out the construction of the cluster structure on $\Conf_3 \A$. We describe the quiver as well as the cluster variables, both in terms of generalized minors and in terms of tensor invariants. We hope that this can serve as a useful summary of the main facts about the cluster structure on $\Conf_3 \A$.

In Section 4, we explain how the structure of the quiver associated to a reduced word for $w_0$ determines the cluster variables as tensor invariants. In Section 5, we verify the conjectures of \cite{Le2}, which imply that the quiver for $\Conf_3 \A$ is determined uniquely given a choice of a reduced word for $w_0$. In Section 6 we extend our results to explain how a double reduced-word for $(w_0,w_0)$ gives a cluster structure on $\Conf_4 \A$.

\medskip
\noindent{\bf Acknowledgments} 
This research was peformed as part of the Perimeter Institute's Undergraduate Research Program.

\section{Background}

\subsection{The definition of $\Conf_m \A$}

Let $S$ be a compact oriented surface with boundary, and possibly with a finite number of marked points on each boundary component. We will always take $S$ to be hyperbolic, meaning it either has negative Euler characteristic, or contains enough marked points on the boundary (in other words, we can give it the structure of a hyperbolic surface such that the boundary components that do not contain marked points are cusps, and all the marked points are also cusps). If we consider the boundary of $S$ and remove the marked points on the boundary, we will call the resulting set the \emph{punctured boundary}.

Let $G$ be a semi-simple algebraic group. When $G$ is simply-connected, we can define the higher Teichmuller space $\A_{G,S}$. It will be the space of twisted local systems on $S$ with structure group $G$ and some extra structure of a framing of the local system at the boundary components of $S$.

The definition of a twisted local system is somewhat technical, so we refer the reader to \cite{FG1} or \cite{Le} for details. Let us give a brief incomplete explanation.

The maximal length element $w_0$ of the Weyl group of $G$ has a natural lift to $G$, denoted $\overline w_0$. Let $s_G:= {\overline w}^2_0$. It turns out that $s_G$ is in the center of $G$ and that $s^2_G =e$. Depending on $G$, $s_G$ will have order one or order two. For example, for $G = SL_{2k}$, $s_G$ has order two, while for $G = SL_{2k+1}$, $s_G$ has order one. In type $C_n$, $s_G$ has order $2$, while in types $B_n$ and $D_n$, the order of $s_G$ depends on $n \mod 4$. A twisted local system is a local system on the punctured tangent bundle of $S$ such that the monodromy in any fiber is a particular central element called $s_G$.

In this paper we will only be concerned with the case where $S$ is a disc with $m$ marked points (in fact we will only deal with the cases $m=3$ or $4$.)  We can describe the space $\A_{G,S}$ quite explicitly in this case. Let us take the components of the punctured boundary of $S$ and number them $1, 2, \dots, m$, starting at some component and moving counterclockwise.

All twisted local systems on a disc are isomorphic, so all the interesting data comes from the framing along the punctured boundary. For each of the components of the boundary, the framing corresponds to a principal flag. Thus we get $m$ elements of $G/U$.  The automorphisms of the twisted local system change these flags by the diagonal action of $G$ on $(G/U)^m$. Thus we get a non-canonical identification
$$\A_{G, S} \simeq \Conf_m \A := G \backslash (G/U)^m.$$
 
If we had chosen a different boundary component to label as $1$, we would have a different identification $\A_{G, S} \simeq \Conf_m \A$ which would differ from our original identification by some power of the twisted cyclic shift map
$$T: \Conf_m \A \rightarrow \Conf_m \A,$$
$$T(U_1, U_2, \dots, U_n) = (s_G \cdot U_n, U_1, \dots, U_{n-1}).$$
In the case of a disc with $m$ marked points, we will think of $\A_{G, S}$ as $\Conf_m \A$ equipped with the cyclic shift map.

Note that the appearance of $s_G$ in the twisted cyclic shift maps comes from the fact that we are dealing with \emph{twisted} local systems. Note that this is essential, as the cluster structure on $\Conf_m \A$ is equivariant under $T$, the twisted cyclic shift map, but not under the normal cyclic shift map.

\subsection{Cluster algebras} \label{cluster}

We review here the basic definitions of cluster algebras. Cluster algebras are commutative rings that come equipped with a collection of distinguished sets of generators, called \emph{cluster variables}. These are sometimes called $\A$-coordinates, as they are functions on the $\A$ space. Each set of generators forms a \emph{cluster}. Starting from an initial cluster, one generates other clusters by the process of \emph{mutation}.

Each set of generators belongs to a seed, which roughly consists of the set of generators along with a $B$-matrix, which encodes how one mutates from one seed to any adjacent seed.

Cluster algebras are determined by an initial seed. A seed $\Sigma = (I,I_0,B,d)$ consists of the following data: 
\begin{enumerate}
\item An index set $I$ with a subset $I_0 \subset I$ of ``frozen'' indices. 
\item A rational $I \times I$ \emph{exchange matrix} $B$. It should have the property that $b_{ij} \in \Z$ unless both $i$ and $j$ are frozen.  
\item A set $d = \{d_i \}_{i \in I}$ of positive integers that skew-symmetrize $B$: $$b_{ij}d_j = -b_{ji}d_i$$ for all $i,j \in I.$ The integers $d_i$ are called \emph{multipliers}.
\end{enumerate}

For most purposes, the values of $d_i$ are only important up to simultaneous scaling. Also note that the values of $b_{ij}$ where $i$ and $j$ are both frozen will play no role in the cluster algebra, though it is sometimes convenient to assign values to $b_{ij}$ for bookkeeping purposes. These values become important in \emph{amalgamation}, where one unfreezes some of the frozen variables.

We can encode a seed via a quiver. In this paper, we will only encounter seeds where the values of the $d_i$ take on two values: we will either have that $d_i \in \{1, 2\}$ or $d_i \in \{1, 3\}$. We can color the vertices so that a vertex is black if $d_i=1$ and white if $d_i=2$ or $3$.

The arrows in the quiver carry the following information:
\begin{itemize}
\item An arrow from $j$ to $i$ means that $b_{ij}>0$ and $b_{ji}<0$.
\item $|b_{ij}|=2$ if $d_i=2$ and $d_j=1$.
\item $|b_{ij}|=3$ if $d_i=3$ and $d_j=1$.
\item $|b_{ij}|=1$ otherwise.
\end{itemize}

%Between frozen vertices, we will allow dotted arrows, which are worth half an undotted arrow, so that, for example, $|b_{ij}|=1$ if $d_i=2$ and $d_j=1$, $|b_{ij}|=\frac{3}{2}$ if $d_i=3$ and $d_j=1$, and $|b_{ij}|=\frac{1}{2}$ otherwise.

A dotted arrow between vertices $i$ and $j$ means that $b_{ij}$ is half what it would be if the arrow were solid. Thus dotted arrows are ``half-arrows.'' Our quivers will usually not carry the information of which vertices are frozen; one should just remember which vertices are frozen.

Let $k \in I \setminus I_0$ be an unfrozen index of a seed $\Sigma$. Then we can perform a mutation at $k$ to get another seed $\Sigma' = \mu_k(\Sigma)$. The frozen variables and $d_i$ are preserved, and the exchange matrix $B'$ of $\Sigma'$ satisfies
\begin{align}\label{eq:matmut}
b'_{ij} = \begin{cases}
-b_{ij} & i = k \text{ or } j=k \\
b_{ij} & b_{ik}b_{kj} \leq 0 \\
b_{ij} + |b_{ik}|b_{kj} & b_{ik}b_{kj} > 0.
\end{cases}
\end{align}

To a seed $\Sigma$ we associate a collection of \emph{cluster variables} $\{A_i\}_{i \in I}$ and a split algebraic torus $\A_\Sigma := \Spec \Z[A^{\pm 1}_I]$.

If $\Sigma'$ is obtained from $\Sigma$ by mutation at $k \in I \setminus I_0$, there is a birational \emph{cluster transformation} $\mu_k: \A_\Sigma \to \A_{\Sigma'}$.  This map is given by the \emph{exchange relation}

\begin{align}\label{eq:Atrans}
\mu_k^*(A'_i) = \begin{cases}
A_i & i \neq k \\
A_k^{-1}\biggl(\prod_{b_{kj}>0}A_j^{b_{kj}} + \prod_{b_{kj}<0}A_j^{-b_{kj}}\biggr) & i = k.
\end{cases}
\end{align} 

\section{The cluster structure on $\Conf_3 \A_G$}

We start by fixing some notation. $G$ will be a simply-connected, semi-simple group. Let us fix a Borel subgroup $B=B^+ \subset G$ and an opposite Borel subgroup $B^-$. Let $\Phi$ be the set of roots, $\Phi^+ \subset \Phi$ be the set of positive roots, and $\alpha_1,\dots,\alpha_n$ be the set of simple roots. The simple roots correspond to nodes in the Dynkin diagram for $G$. Let $\omega_1,\dots,\omega_n$ be the fundamental weights. They form a basis orthogonal to the $\alpha_i$, so that $\langle \omega_i,\alpha_j\rangle =\delta_{ij}$.

Let $W$ be the Weyl group of $G$. It is generated by the simple reflections $s_1,\dots,s_n$ which correspond to reflections in the hyperplanes perpendicular to $\alpha_1,\dots,\alpha_n$. Any element $w$ of the Weyl group has a length, which is the minimal length of an expression of $w$ in terms of the generators $s_1,\dots,s_n$. Denote by $w_0$ the unique element of $W$ of maximal length $K:=|\Phi^+|=\frac{|\Phi|}{2}$. It is a fact that $w_0$ must correspond to $-\Psi$ for some automorphism $\Psi$ of the Dynkin diagram $D$ (which is often the identity).

We will call an expression
$$w_0=s_{i_1} s_{i_2} s_{i_3} \cdots s_{i_{K-1}} s_{i_K}$$
a \emph{reduced word for $w_0$}.

The goal of this section will be to give a procedure for constructing the cluster structure on $\Conf_3 \A_G$. The cluster structure on $\Conf_3 \A_G$ is closely related to the cluster structure on $B^-$. (It is also, of course, related to the cluster structure on $B^+$, but we choose to work with $B^-$.) 

Berenstein, Fomin and Zelevinsky constructed cluster structure on $B^-$, the Borel in the group $G$ (\cite{BFZ}). They construct a seed for each choice of a reduced word for $w_0$. They explicitly construct a quiver, and define the cluster variables as \emph{generalized minors}.

 In fact, we will see that each seed for $B^-$ coming from a reduced word for $w_0$ is a sub-seed of a corresponding seed for $\Conf_3 \A_G$. Let us elaborate.

Let $(A_1, A_2, A_3)$ be a triple of principal flags in $\Conf_3 \A_G$. There will be functions attached to these vertices that only depend on two of the three flags. We will call these the {\em edge functions} or \emph{edge variables}. Let us call all other functions, which depend on all three flags, {\em face functions} or \emph{face variables}. We are interested in describing the face functions, as well as the edge functions attached to the edges $A_1A_2$, $A_2A_3$, and $A_1A_3$.

Consider the natural map from $B^-$ to $\Conf_3 \A_G$ given by the formula
$$i: b \in B^- \rightarrow (U^-, \overline{w_0}U^-, b\cdot \overline{w_0}U^-) \in \Conf_3 \A_{G}.$$
This is an injective map. Functions on $\Conf_3 \A_G$ can be pulled back to give functions on $B^-$. We will describe a cluster structure on $\Conf_3 \A_G$ such that cluster variables on $B^-$ are pulled back from some subset of cluster variables on $\Conf_3 \A_G$. We can be somewhat more precise: the face variables and the edge variables for the edges $A_2A_3$, and $A_1A_3$ pull back to give cluster variables on $B^-$, while the edge variables for the edge $A_1A_2$ take the value $1$ on the image $i(B^-)$.

\subsection{The cluster algebra on $B^-$}

Now let us review how to construct the cluster structure on $B^-$. A more detailed account can be found in \cite{BFZ}.

The double Bruhat cell 
$$G^{w_0,e}:= B^+w_0B^+ \cap B^-eB^-$$
is an open set in $B^-$. The image of $G^{w_0,e} \subset B^-$ under $i$ is the set of triples of flags where every pair is in generic position, and the edge variables on the edge $A_1A_2$ take the value $1$.

Take any reduced-word $s_{i_1}\dots s_{i_K}$ for $w_0$. Our convention will be that we read the word from right to left, i.e., the simple reflection $s_K$ followed by the simple reflection $s_{K-1}$, etc. Here $1 \leq i_j \leq n$.

For each reduced-word expression for $w_0$ there is a corresponding seed for the cluster algebra on $B^-$. The $B$-matrix for this seed can encoded via a quiver. This quiver will have $n+K$ vertices, of which $2n$ are frozen edge vertices. Each vertex of the quiver is associated to one of the nodes of the Dynkin diagram. If the simple reflection $s_i$ occurs $a_i$ times in the reduced-word for $w_0$, there will be $a_i+1$ vertices belonging to the node $i$. Of these, two vertices (the first and the last) will be frozen.

We will arrange the vertices of the quiver in horizontal rows, where each row is labelled by a node in the Dynkin diagram, and all the vertices in that row are associated to that node in the Dynkin diagram.

The multipliers $d_i$ for the vertices are determined by the nodes they are associated with. If a node in the Dynkin diagram is associated with a short root, the multipliers for the vertices belong to this node are $1$. If a node in the Dynkin diagram is associated with a long root, the multipliers for the vertices belong to this node are $2$ in types $B, C, F$ and $3$ in type $G$. Alternatively, let the node $i$ correspond to the root $\alpha_i$, and let us normalize the lengths of the roots so that the short roots have length $1$. Then the multipliers for the vertices belonging to node $i$ are $(\textrm{length of } \alpha_i)^{2}$.

We have described the vertices of the quiver and the multipliers for these vertices. To describe the arrows of the quiver, we will glue the quiver together out of smaller pieces, one for each simple reflection $s_{i_j}$ occuring in our reduced word for $w_0$.

The example of $SL_4$ well serve as a running example.

One reduced word for $w_0 \in W_{SL_4}$ is 
$$s_1 s_2 s_1 s_3 s_2 s_1.$$
The corresponding quiver is for $B^-$ is

\begin{center}
\begin{tikzpicture}[scale=2]

  \node (x310) at (-0.5,0.9) {$\bullet$};
  \node (x301) at (0.5,0.9) {$\bullet$};

  \node (x220) at (-1,0) {$\bullet$};
  \node (x211) at (0,0) {$\bullet$};
  \node (x202) at (1,0) {$\bullet$};

  \node (x130) at (-1.5,-0.9) {$\bullet$};
  \node (x121) at (-0.5,-0.9) {$\bullet$};
  \node (x112) at (0.5,-0.9) {$\bullet$};
  \node (x103) at (1.5,-0.9) {$\bullet$};

%  \node (x031) at (-1,-1.8) {$\bullet$};
%  \node (x022) at (0,-1.8) {$\bullet$};
%  \node (x013) at (1,-1.8) {$\bullet$};

  \draw [->, postaction={decorate}] (x301) -- (x310) node [midway, above] {$s_3$};
  \draw [->, postaction={decorate}] (x202) -- (x211) node [midway, above] {$s_2$};
  \draw [->, postaction={decorate}] (x211) -- (x220) node [midway, above] {$s_2$};
  \draw [->, postaction={decorate}] (x103) -- (x112) node [midway, above] {$s_1$};
  \draw [->, postaction={decorate}] (x112) -- (x121) node [midway, above] {$s_1$};
  \draw [->, postaction={decorate}] (x121) -- (x130) node [midway, above] {$s_1$};
%  \draw [->, dashed] (x013) to (x022);
%  \draw [->, dashed] (x022) to (x031);

%  \draw [->] (x013) to (x103);
%  \draw [->] (x022) to (x112);
  \draw [->] (x112) to (x202);
%  \draw [->] (x031) to (x121);
  \draw [->] (x121) to (x211);
  \draw [->] (x211) to (x301);
  \draw [->, dashed] (x130) to (x220);
  \draw [->, dashed] (x220) to (x310);

%  \draw [->] (x130) to (x031);
  \draw [->] (x220) to (x121);
%  \draw [->] (x121) to (x022);
  \draw [->] (x310) to (x211);
  \draw [->] (x211) to (x112);
%  \draw [->] (x112) to (x013);
  \draw [->, dashed] (x301) to (x202);
  \draw [->, dashed] (x202) to (x103);

\draw[yshift=-1.5cm]
  node[below,text width=6cm] % ,style=information text --- don't know this one
  {
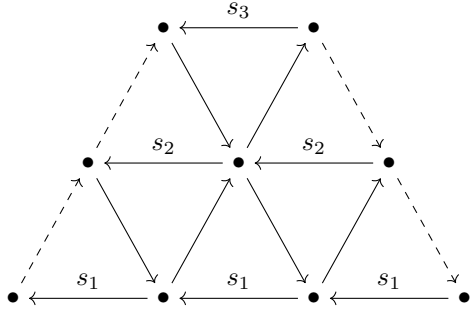
  Figure 1. The quiver corresponding to the reduced word $s_1 s_2 s_1 s_3 s_2 s_1$ for $B^-_{SL_{4}}$.
  };

\end{tikzpicture}
\end{center}

In the above quiver, there are three rows. The vertices belonging to node $1$ of the Dynkin diagram are on the bottom row; the vertices belonging to node $2$ are in the middle row; and the vertices belonging to node $3$ are on the top row. $SL_n$ is simply-laced, so $d_i=1$ for all vertices. The vertices on either end of each row are frozen vertices.

The leftmost vertices in each row are pulled back from the edge functions for the edge $A_1A_3$. The rightmost vertices in each row are pulled back from the edge functions for the edge $A_2A_3$. Observe that dotted arrows only go between frozen vertices, as required. The flags $A_1,A_2, A_3$ are oriented as follows:

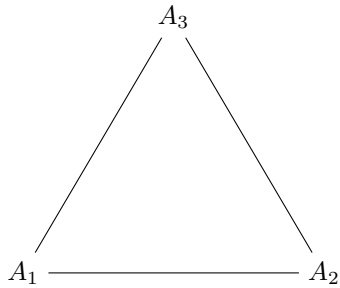
\begin{center}
\begin{tikzpicture}[scale=2]

  \node (A_1) at (-1,0) {$A_1$};
  \node (A_2) at (1,0) {$A_2$};
  \node (A_3) at (0,1.7) {$A_3$};

  \draw [] (A_1) -- (A_2);
  \draw [] (A_2) -- (A_3);
  \draw [] (A_3) -- (A_1);

\end{tikzpicture}
\end{center}

The larger quiver is glued out of smaller pieces, each piece corresponding to an occurence of one of the simple reflections $s_1, s_2, s_3$ in the reduced word for $w_0$. These pieces are depicted in Figure 2:

\begin{center}
\begin{tikzpicture}[scale=2]

  \node (x211) at (-2,0.9) {\Large $\ontop{2}{\bullet}$};

  \node (x121) at (-2.5,0) {\Large $\ontop{1_-}{\bullet}$};
  \node (x112) at (-1.5,0) {\Large $\ontop{1_+}{\bullet}$};
%  \node (12k) at (-2,-1.5) {\Large $\ontop{k^{\circ}}{\bullet}$};

  \draw [->, postaction={decorate}] (x112) -- (x121) node [midway, above] {$s_1$};
  \draw [->, dashed] (x121) to (x211);
  \draw [->, dashed] (x211) to (x112);
%  \draw [->] (x121) to (12k);
%  \draw [->] (12k) to (x112);

  \node (x310) at (0,0.9) {\Large $\ontop{3}{\bullet}$};

  \node (x220) at (-0.5,0) {\Large $\ontop{2_-}{\bullet}$};
  \node (x211) at (0.5,0) {\Large $\ontop{2_+}{\bullet}$};

  \node (x121) at (0,-0.9) {\Large $\ontop{1}{\bullet}$};
%  \node (12k) at (0,-1.5) {\Large $\ontop{k^{\circ}}{\bullet}$};

  \draw [->, postaction={decorate}] (x211) -- (x220) node [midway, above] {$s_2$};

  \draw [->, dashed] (x121) to (x211);
  \draw [->, dashed] (x220) to (x310);

  \draw [->, dashed] (x220) to (x121);
  \draw [->, dashed] (x310) to (x211);

%  \draw [->] (12k) to (x211);
%  \draw [->] (x220) to (12k);

  \node (x310) at (1.5,0) {\Large $\ontop{3_-}{\bullet}$};
  \node (x301) at (2.5,0) {\Large $\ontop{3_+}{\bullet}$};

  \node (x211) at (2,-0.9) {\Large $\ontop{2}{\bullet}$};
%  \node (12k) at (2,-1.5) {\Large $\ontop{k^{\circ}}{\bullet}$};

  \draw [->, postaction={decorate}] (x301) -- (x310) node [midway, above] {$s_3$};
  \draw [->,dashed] (x211) to (x301);
  \draw [->,dashed] (x310) to (x211);
  
%  \draw [->] (12k) to (x301);
%  \draw [->] (x310) to (12k);

\draw[yshift=-1.2cm]
  node[below,text width=6cm] % ,style=information text --- don't know this one
  {
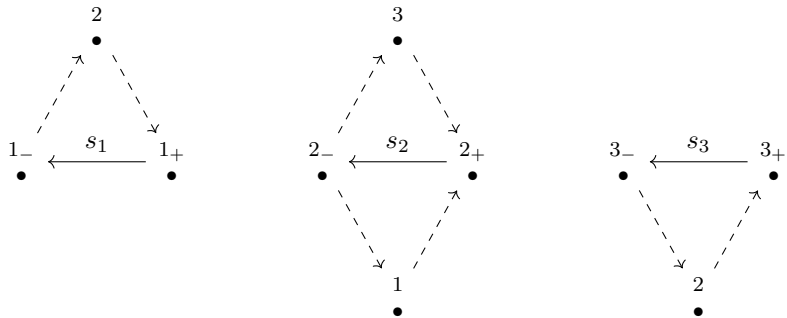
  Figure 2. Piece of the quiver corresponding to the simple reflections $s_1, s_2, s_3$.
  };

\end{tikzpicture}
\end{center}

We glue these these pieces together to obtain the quiver for $B^-$.

Let us now describe the pieces in general. Consider the simple reflection $s_j$. A piece of the quiver will consist of $n+1$ vertices. There will be two vertices $j_-$ and $j_+$ on the $j$-th row (associated to the $j$-th node of the Dynkin diagram). We will picture $j_+$ to the right of $j_-$. There will be one vertex $i$ on the $i$-th row for each every other node $i \neq j$ of the Dynkin diagram. The quiver will have the following arrows:

\begin{itemize}
\item An arrow from $j_+$ to $j_-$.
\item A dotted arrow from $j_-$ to $i$ whenever $i$ and $j$ are adjacent in the Dynkin diagram.
\item A dotted arrow from $i$ to $j_+$ whenever $i$ and $j$ are adjacent in the Dynkin diagram.
\end{itemize}

It is easy to check that for $SL_4$, we get the pieces corresponding to $s_1, s_2, s_3$ as pictured above. In Figure 2 above, we have not pictured vertices that have no incoming or outgoing arrows.

We can now describe a quiver associated to any reduced word $u$, for any $u \in W$. We use the following rule: Suppose that $u=u_1u_2$. Then we take the quiver for $u_1$ and put it to the left of the quiver for $u_2$, then identify the rightmost vertices of the quiver for $u_1$ with the leftmost vertices of the quiver for $u_2$. When we perform this gluing, two dotted arrows in the same direction glue to give us a solid arrow, whereas two dotted arrows in the opposite direction cancel to give us no arrow. (In the $SL_4$ example above, it turns out we never glue dotted arrows going in the opposite direction, but this does happen in general.) This is the process of amalgamation. We use this procedure to get a quiver for any reduced word for $w_0$.

Let us now describe the functions attached to the vertices of the quiver. The functions are given by \emph{generalized minors} of $B^-$, which we now define. Let $G_0 = U^- H U^+ \subset G$ be the open subset of $G$ consisting of elements that have a Gaussian decomposition $x=[x]_- [x]_0 [x]_+$. Then for any two elements $u, v \in W$, and any fundamental weight $\omega_i$, we have the {\em generalized minor} $\Delta_{u\omega_i, v\omega_i}(x)$ defined by

$$\Delta_{u\omega_i, v\omega_i}(x) := ([\overline{u}^{-1}x \overline{v}]_0)^{\omega_i}.$$
The formula gives a well-defined value when $\overline{u}^{-1}x \overline{v} \in G_0$, but may have poles elsewhere.

We start with a reduced word for $w_0$:
$$w_0=s_{i_1} s_{i_2} s_{i_3} \cdots s_{i_{K-1}} s_{i_K}$$
For $1 \leq l \leq K$, we have the subword $$u_l:=s_{i_1}\dots s_{i_l}.$$ One can take $u_0=e$.
In our situation, we are interested in the generalized minors $\Delta_{u\omega_i, v\omega_i}(x)$ when $v=e$ and $u=u_{l}=s_{i_1} s_{i_2} s_{i_3} \cdots s_{i_l}$ for $0 \leq l \leq K$.

Write $w_0=u_l u_l^c$. Then we can amalgamate the quiver from two parts: $u_l$ on the left and $u_l^c$ on the right. The whole quiver is obtained from indentifying the rightmost vertices of the quiver for $u_l$ with the leftmost vertices for the quiver for $u_l^c$. There are $n$ vertices which get identified, one on each level. To these vertices we attach the function 
$$\Delta_{u_{l}\omega_{i}, \omega_{i}}.$$
for $i=1, \dots, n$. Many vertices in the quiver will be defined by more than one generalized minor, but the expressions will all agree.

%Suppose that in the reduced word for $w_0$ that we have chosen, there are $a_i$ occurences of the simple reflection $s_i$. Then the quiver contains $a_i+1$ vertices belonging to the node $i$ of the Dynkin diagram. Let us call these vertices $x_{i0}, x_{i1}, x_{i2}, \dots, x_{ia_i}$. If $s_{i_l}$ is the $j^{th}$ occurence of the simple reflection $s_i$ in the reduced word for $w_0$, then the function $\Delta_{u_{l}\omega_{i_l}, \omega_{i_l}},$ will be attached to the vertex $x_{ij}$.

The leftmost function in each row will be $\Delta_{\omega_i, \omega_i}$. These are the edge functions for the edge $A_1A_3$. The rightmost function in each row will be $\Delta_{w_0\omega_i, \omega_i}$. These are the edge functions for the edge $A_2A_3$. All the remaining functions are face functions.

In our example of $G=SL_4$, the generalized minors are attached to the vertices of the quiver as follows:

\begin{center}
\begin{tikzpicture}[scale=2.4]

  \node (x103) at (-0.5,0.9) {$\Delta_{\omega_3, \omega_3}$};
  \node (x013) at (0.5,0.9) {$\Delta_{s_1 s_2 s_1 s_3\omega_3, \omega_3}$};

  \node (x202) at (-1,0) { $\Delta_{\omega_2, \omega_2}$};
  \node (x112) at (0,0) { $\Delta_{s_1 s_2\omega_2, \omega_2}$};
  \node (x022) at (1,0) { $\Delta_{s_1 s_2 s_1 s_3 s_2\omega_2, \omega_2}$};

  \node (x301) at (-1.5,-0.9) { $\Delta_{\omega_1, \omega_1}$};
  \node (x211) at (-0.5,-0.9) { $\Delta_{s_1\omega_1, \omega_1}$};
  \node (x121) at (0.5,-0.9) { $\Delta_{s_1 s_2 s_1\omega_1, \omega_1}$};
  \node (x031) at (1.5,-0.9) { $\Delta_{w_0\omega_1, \omega_1}$};

  \draw [->] (x013) to (x103);
  \draw [->] (x022) to (x112);
  \draw [->] (x112) to (x202);
  \draw [->] (x031) to (x121);
  \draw [->] (x121) to (x211);
  \draw [->] (x211) to (x301);

  \draw [->] (x121) to (x022);
  \draw [->] (x211) to (x112);
  \draw [->] (x112) to (x013);
  \draw [->, dashed] (x301) to (x202);
  \draw [->, dashed] (x202) to (x103);

  \draw [->] (x202) to (x211);
  \draw [->] (x103) to (x112);
  \draw [->] (x112) to (x121);
  \draw [->, dashed] (x013) to (x022);
  \draw [->, dashed] (x022) to (x031);

\draw[yshift=-1.5cm]
  node[below,text width=6cm] % ,style=information text --- don't know this one
  {
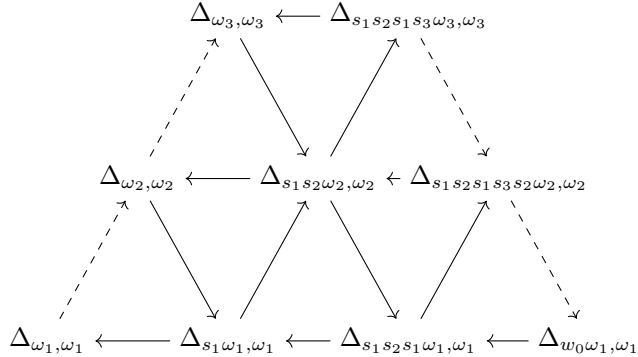
  Figure 3. Generalized minors give the cluster variables for $B^- \subset SL_4$.
  };

\end{tikzpicture}
\end{center}

\subsection{The cluster structure on $\Conf_3 \A_G$}

To extend the cluster structure on $B^-$ to $\Conf_3 \A_G$, we first need to realize generalized minors as being pulled back from functions on $\Conf_3 \A_G$. Recall that the functions on $\Conf_3 \A_G$ are given by invariants of triple tensor products:
$$\mathcal{O}(\Conf_3 \A_G) \simeq \bigoplus [V_{\lambda} \otimes V_{\mu} \otimes V_{\nu}]^G,$$
where the sum is over all triples of dominant weights $(\lambda, \mu, \nu)$.

A direct calculation shows that the functions along the edges $A_1A_3$ and $A_2A_3$ are given by the canonical invariant in the tensor products
$$[V_{\omega_i} \otimes \mathbbm{1} \otimes V_{\omega_i^{*}}]^G,$$
$$[\mathbbm{1} \otimes V_{\omega_i} \otimes V_{\omega_i^{*}}]^G,$$
respectively. Here $\omega_i$ are the fundamental weights, and for any weight $\lambda$, we define $\lambda^{*}$ to be the weight corresponding to the representation dual to $V_{\lambda}$. In other words, $\lambda^{*}=-w_0(\lambda)$. Here, $\mathbbm{1}$ is, of course, the trivial representation.

Naturally, we will take the edge functions for the edge $A_1A_2$, to be invariants in
$$[V_{\omega_i} \otimes V_{\omega_i^{*}} \otimes \mathbbm{1}]^G.$$

Now let us deal with the face functions. For any $u \in W$, we can write
$$u\omega_i = \sum r_i \omega_i,$$
for some integers $r_i$.

For an integer $r$, define $r^+=\max(0,r)$ and $r^-=\min(0,r)$. For any weight $\lambda = \sum r_i \omega_i$, we can define $\lambda^+ = \sum r_i^+ \omega_i$ and $\lambda^- = \sum r_i^- \omega_i$.

Let us set $\lambda = -w_0 (u\omega_i)^+$ and $\mu = -(u\omega_i)^-$. Then we will have that $u\omega_i = -w_0 \lambda - \mu.$

A more involved calculation gives the following formula, which will be important for us:
\begin{prop} \label{grading} For any $u \in W$,
$$\Delta_{u \omega_{i}, \omega_{i}} \in  [V_{\lambda} \otimes V_{\mu} \otimes V_{\omega_i}]^G$$
where $\lambda$ and $\mu$ are chosen as above so that 
$u\omega_i = -w_0 \lambda - \mu.$
\end{prop}

\begin{proof} This formula was verified for groups $G$ of type $A, B, C, D, G$ in \cite{Le}, \cite{Le2}. If one checks the formula for one reduced word for $w_0$ for a group $G$, it is simple to check that the formula is stable under mutation, and hence holds for all reduced words. A general proof will appear in work of Goncharov and Shen.
\end{proof}

Let us say a few words about this. There are many functions on $\Conf_3 \A_G$ that can be pulled back to give the function $\Delta_{u \omega_{i}, \omega_{i}}$. 

On the image of $B^-$, the edge functions for the edge $A_1A_2$ all take the value $1$. Thus we can multiply the tensor invariant in
$$[V_{\lambda} \otimes V_{\mu} \otimes V_{\omega_i}]^G$$
which pulls back to $\Delta_{u \omega_{i}, \omega_{i}}$ by any combination of those edge functions. As a result we can obtain, for any dominant weight $\nu$, a tensor invariant in
$$[V_{\lambda + \nu} \otimes V_{\mu+ \nu^*} \otimes V_{\omega_i}]^G$$ which also pulls back to $\Delta_{u \omega_{i}, \omega_{i}}$. In fact, all functions on $\Conf_3 \A_G$ which are tensor invariants and pull back to $\Delta_{u \omega_{i}, \omega_{i}}$ are obtained in this way. The tensor invariant in $[V_{\lambda} \otimes V_{\mu} \otimes V_{\omega_i}]^G$ is therefore the tensor invariant of minimal weight which pulls back to give the generalized minor $\Delta_{u \omega_{i}, \omega_{i}}$.

We can now define all the functions in the cluster structure for $\Conf_3 \A_G$. The edge variables will be given by the canonical invariants in 
$$[V_{\omega_i} \otimes \mathbbm{1} \otimes V_{\omega_i^{*}}]^G,$$
$$[\mathbbm{1} \otimes V_{\omega_i} \otimes V_{\omega_i^{*}}]^G,$$
$$[V_{\omega_i} \otimes V_{\omega_i^{*}} \otimes \mathbbm{1}]^G.$$
The generalized minors will extend to $\Conf_3 \A_G$ by taking the tensor invariant in $$[V_{\lambda} \otimes V_{\mu} \otimes V_{\omega_i}]^G$$
which pulls back to $\Delta_{u \omega_{i}, \omega_{i}}$ as we described.

We now need to describe the quiver for the cluster structure on $\Conf_3 \A_G$. The quiver for $B^-$ gives us all the multipliers and arrows for the vertices corresponding to face functions or edge functions for the edges $A_2A_3$ and $A_1A_3$.

Let us specify the multipliers for the edge $A_1A_2$. The multiplier for the function in the invariant space $[V_{\omega_i} \otimes V_{\omega_i^{*}} \otimes \mathbbm{1}]^G$ is given by the multiplier for the $i$-th node of the Dynkin diagram, which we recall was $(\textrm{length of } \alpha_i)^{2}$.

We now give a description of the quiver. Recall that the quiver for $B^-$ was glued together out of pieces in a way specified by a choice of a reduced word $w_0=s_{i_1} s_{i_2} s_{i_3} \cdots s_{i_{K-1}} s_{i_K}$.
Recall that we also have the sequence of subwords $u_l:=s_{i_1}\dots_{i_l}$ for $0 \leq l \leq K$. Recall that the Weyl group permutes the set of roots of $G$. Suppose that $\alpha$ is a positive root. When we consider the image of $\alpha$ under the sequence $e=u_0^{-1}, u_1^{-1}, \dots, u_K^{-1} = w_0^{-1}=w_0$, each root in the sequence we obtain is either positive or negative.

\begin{fact} If $\alpha$ is a positive root, the sequence of roots $u_i^{-1}\alpha$ starts out positive, and once a negative root occurs, the rest of the roots are negative. Moreover, if we consider the action of the sequence $u_i^{-1}$ on the set of positive roots, at each stage, exactly one root switches from being positive to negative.
\end{fact}

Thus we can associate to a reduced word  $w_0=s_{i_1} s_{i_2} s_{i_3} \cdots s_{i_{K-1}} s_{i_K}$ a sequence of positive roots (note that $K$ is also the number of positive roots.) Now we can single out those indices $t_1, t_2, \dots, t_n$ at which a \emph{simple} root goes from being positive to negative. Take $t_k$ to be the smallest integer such that $u_{i_{t_k}}^{-1}\alpha_k$ is negative.

The quiver for $\Conf_3 \A_G$ is glued together out of the same pieces as before, except that the pieces corresponding to $s_{i_{t_k}}$ are modified. Suppose that $s_{i_{t_k}} = s_j$. Then the quiver for this simple reflection will consist of $n+2$ vertices. There will be two vertices $j_-$ and $j_+$ on the $j$-th row as before. There will be one vertex $i$ on the $i$-th row for each every other node $i \neq j$ of the Dynkin diagram. Moreover, the quiver will have a vertex $k^{\circ}$, which will correspond to the edge function for the edge $A_1A_2$ in the invariant space $[V_{\omega_k^{*}} \otimes V_{\omega_k} \otimes \mathbbm{1}]^G$

The quiver will have the following arrows:

\begin{itemize}
\item An arrow from $j_+$ to $j_-$.
\item A dotted arrow from $j_-$ to $i$ whenever $i$ and $j$ are adjacent in the Dynkin diagram.
\item A dotted arrow from $i$ to $j_+$ whenever $i$ and $j$ are adjacent in the Dynkin diagram.
\item An arrow from $j_-$ to $k^{\circ}$.
\item An arrow from $k^{\circ}$ to $j_+$.
\end{itemize}

Here is how the pieces $s_1, s_2, s_3$ would be modified if they occured as $s_{i_{t_k}}$:

\begin{center}
\begin{tikzpicture}[scale=2]

  \node (x211) at (-2,0.9) {\Large $\ontop{2}{\bullet}$};

  \node (x121) at (-2.5,0) {\Large $\ontop{1_-}{\bullet}$};
  \node (x112) at (-1.5,0) {\Large $\ontop{1_+}{\bullet}$};
  \node (12k) at (-2,-1.5) {\Large $\ontop{k^{\circ}}{\bullet}$};

  \draw [->, postaction={decorate}] (x112) -- (x121) node [midway, above] {$s_1$};
  \draw [->, dashed] (x121) to (x211);
  \draw [->, dashed] (x211) to (x112);
  \draw [->] (x121) to (12k);
  \draw [->] (12k) to (x112);

  \node (x310) at (0,0.9) {\Large $\ontop{3}{\bullet}$};

  \node (x220) at (-0.5,0) {\Large $\ontop{2_-}{\bullet}$};
  \node (x211) at (0.5,0) {\Large $\ontop{2_+}{\bullet}$};

  \node (x121) at (0,-0.9) {\Large $\ontop{1}{\bullet}$};
  \node (12k) at (0,-1.5) {\Large $\ontop{k^{\circ}}{\bullet}$};

  \draw [->, postaction={decorate}] (x211) -- (x220) node [midway, above] {$s_2$};

  \draw [->, dashed] (x121) to (x211);
  \draw [->, dashed] (x220) to (x310);

  \draw [->, dashed] (x220) to (x121);
  \draw [->, dashed] (x310) to (x211);

  \draw [->] (12k) to (x211);
  \draw [->] (x220) to (12k);

  \node (x310) at (1.5,0) {\Large $\ontop{3_-}{\bullet}$};
  \node (x301) at (2.5,0) {\Large $\ontop{3_+}{\bullet}$};

  \node (x211) at (2,-0.9) {\Large $\ontop{2}{\bullet}$};
  \node (12k) at (2,-1.5) {\Large $\ontop{k^{\circ}}{\bullet}$};

  \draw [->, postaction={decorate}] (x301) -- (x310) node [midway, above] {$s_3$};
  \draw [->,dashed] (x211) to (x301);
  \draw [->,dashed] (x310) to (x211);
  
  \draw [->] (12k) to (x301);
  \draw [->] (x310) to (12k);

\draw[yshift=-2cm]
  node[below,text width=6cm] % ,style=information text --- don't know this one
  {
  Figure 4. Modified pieces of the quiver for $s_1, s_2, s_3$ when a simple root changes sign for $\Conf_3 \A_G$.
  };

\end{tikzpicture}
\end{center}

As before, we can amalgamate the pieces corresponding to the simple reflections in the reduced word decomposition for $w_0$, making sure to use the modified pieces for the $s_{i_{t_k}}$. The last step is that we need to add some arrows between the vertices $k^{\circ}$. We add a dotted arrow from $k^{\circ}$ to $k'^{\circ}$ whenever nodes $k$ and $k'$ are adjacent in the Dynkin diagram for $G$, and $t_{k'} < t_k$. This completes the construction of the quiver for $\Conf_3 \A_G$.

\section{The quiver determines tensor invariants}

In this section, we explain how the structure of the quiver for $\Conf_3 \A_G$, as described above via the process of amalgamation, determines the tensor invariant spaces that the cluster variables lie in.

Suppose that $(A_1,A_2,A_3) \in \Conf_3 \A_G$ is a configuration of three principal flags. Let $\B_G : = G/B$ denote the usual flag variety. There is a natural map $\pi: \A_G \rightarrow \B_G$. Let $B_i=\pi(A_i)$. A pair consisting of a principal flag and a regular flag which are in generic position with respect to each other determines a frame for the group $G$. Thus we may consider two frames $(A_3.B_1)$ and $(A_3,B_2)$. Let us arrange so that $B_2$ is the standard flag (i.e., it is stabilized by $B^+$) and $A_3$ is the standard opposite flag (it is stabilized by $U^-$). Then there is a unique element $\gamma \in G$ such that 
$$\gamma \cdot (A_3.B_1) = (A_3,B_2).$$
We wish to compute $\gamma$.

For each letter $s_{i_{l}}$ in the reduced word for $w_0$, let us consider the piece of the quiver corresponding to that letter. As before, take $t_k$ to be the smallest integer such that $u_{i_{t_k}}^{-1}\alpha_k$ is negative. First suppose that $l \neq t_k$ for any $k$. Moreover, suppose that $i_{l}=j$. Then the piece of the quiver consists vertices $j_i$, $j_+$ as well as vertices $i$ for all $i \neq j$. Let $a_{j_-}, a_{j_+}, a_{i}$ be the corresponding cluster variables. Then we can define
\begin{equation} \label{beqn} b_l := \frac{\prod_{i \neq j} a_{i}^{-c_{ij}}}{a_{j_-}a_{j_+}}.
\end{equation}
Here, $c_{ij}$ is the entry of the Cartan matrix, $2\frac{\langle \alpha_i, \alpha_j\rangle}{\langle \alpha_i, \alpha_i \rangle}$. Thus $-c_{ij}=0, 1, 2$ or $3$.

Now suppose that $l = t_k$. The modified quiver has an additional vertex $k^{\circ}$. Then we define
\begin{equation} \label{beqn2} b_l := \frac{a_{k^{\circ}}\prod_{i \neq j} a_{i}^{-c_{ij}}}{a_{j_-}a_{j_+}}.
\end{equation}

Now consider the product
$$\prod_{l} F_{i_l^*}(b_l) = F_{i_1^*}(b_1)F_{i_2^*}(b_2) \cdots F_{i_N^*}(b_N).$$
Here $F_i \in U^-$ is the Cartan generator associated to the $i$-th node in the Dynkin diagram, and if $i$ is a node in the Dynkin diagram corresponding to the fundamental representation $V$, then $i^*$ is the node correponding to the representation $V^*$. We can now give a formula for $\gamma$.

\begin{theorem} \label{factorization} In the notation above, if $\gamma \cdot (A_3.B_1) = (A_3,B_2)$, then
$$\gamma = \prod_{l} F_{i_l^*}(b_l).$$
\end{theorem}

We can also give another equivalent form. Suppose that we arrange that $A_3$ is the standard flag (i.e., it is stabilized by $U^+$) and $B_2$ is the standard opposite flag (it is stabilized by $B^-$). Then there is a unique element $\gamma' \in G$ such that 
$$\gamma' \cdot (A_3.B_2) = (A_3,B_1).$$ 
We have the following formula:
$$\gamma' = E_{i_N}(b_N) \cdots E_{i_1}(b_1).$$

It is straightforward to verify that these formulas are equivalent: Simply use that $\gamma' = \overline{w_0} \gamma^{-1} \overline{w_0}^{-1}$ and $E_{i}(b)=\overline{w_0} F_{i^*}(b)^{-1} \overline{w_0}^{-1}$.

\begin{proof} The factorization formula follows directly from work of Berenstein, Fomin and Zelevinsky, \cite{BFZ}. Let us explain. Recall that there is a map 
$$i: b \in G^{w_0, e} \subset B^- \rightarrow (U^-, \overline{w_0}U^-, b\cdot \overline{w_0}U^-) \in \Conf_3 \A_{G}.$$
The image of the map is those configurations of flags such that the edge variables along the edge $A_1A_2$ (those which we called $a_{k^{\circ}}$ above) are equal to $1$. If we set $a_{k^{\circ}} = 1$ in the formula in Theorem \ref{factorization}, we obtain the formula in \cite{BFZ} for $\gamma$, which the authors call the twist of $b$.

However, there is an $H^3$ action on $\Conf_3 \A_{G}$, and because $\gamma$ only depends on $B_1$ and $B_2$, the action of $H \times H \times e$ leaves $\gamma$ invariant. However, using either one of these $H$ actions, we can always arrange so that $a_{k^{\circ}} = 1$.

Thus we only need to show that the formula \ref{factorization} is invariant under the action of $H \times H \times e$. Note that if we have a function $a \in \cO(\Conf_3 \A_{G})$ such that $a$ lies in a tensor invariant space 
$$a \in [V_{\lambda} \otimes V_{\mu} \otimes V_{\nu}]^G,$$
then the action of $(h_1,h_2,h_3)$ on $a$ is given by multiplication by $\lambda(h_1) \mu(h_2) \nu(h_3)$. We will say that $a$ lies in the $(\lambda, \mu, \nu)$ graded piece of $\cO(\Conf_3 \A_{G})$. Now, let us consider the expressions for $b_l$ in equations \ref{beqn} and \ref{beqn2}. It is sufficient to show that $b_l$ is invariant under the action of $H \times H \times e$. In fact, we can show more. We know how the third copy of $H$ acts, because acting by $h_3$ must conjugate $\gamma$ by $h_3$, so that $b_j$ should be multiplied by $\alpha_{i_j}(h_3)^{-1}$.

\begin{prop}\label{recursive} Let us suppose that $a_{?}$ lies in the graded piece $(\lambda_?, \mu_?, \nu_?)$ for $? = j_-, j_+, i, k^{\circ}$. Then
$$\sum_{i \neq j} c_{ij} (\lambda_i, \mu_i, \nu_i) + (\lambda_{j_-}, \mu_{j_-}, \nu_{j_-}) + (\lambda_{j_+}, \mu_{j_+}, \nu_{j_+}) - (\lambda_{k^{\circ}}, \mu_{k^{\circ}}, \nu_{k^{\circ}})= (0,0,\alpha_j).$$
\end{prop}

We will later show that this proposition follows from the formulas in Proposition \ref{grading}. The proposition immediately gives that our formula for $\gamma$ is invariant under $H \times H \times e$, and moreover transforms correctly under the third copy of $H$.
\end{proof}

Let us now prove the above proposition. Let us suppose that we are dealing with piece of the quiver corresponding to the letter $s_{i_l}$ in the reduced subword $u_l=s_{i_1}\dots s_{i_l}$. Let us label the vertices in this portion of the quiver by $i, j_-, j_+$ and possibly $k^\circ$ as before. 

We have that for $i \neq j$, the function attached to the vertex $i$ is $\Delta_{u_{l}\omega_{i}, \omega_{i}}$, which lies in the graded piece $(\lambda_i, \mu_i, \omega_i)$, where
$$-w_0 \lambda_i = (u_l \omega_i)^+,$$
$$\mu_i = -(u_l \omega_i)^-.$$
Similarly, the function at vertex $j_-$ lies in the graded piece $(\lambda_{j_-}, \mu_{j_-}, \nu_{j_-})$ where
$$-w_0 \lambda_{j_-} = (u_{l-1} \omega_j)^+,$$
$$\mu_{j_-} = -(u_{l-1} \omega_j)^-,$$
and the function at vertex $j_+$ lies in the graded piece $(\lambda_{j_+}, \mu_{j_+}, \nu_{j_+})$ where
$$-w_0 \lambda_{j_+} = (u_{l} \omega_j)^+,$$
$$\mu_{j_+} = -(u_{l} \omega_j)^-.$$

Any given coweight $\omega_{i'}$ will only contribute to one of $-w_0 \lambda_i$ or $\mu_i$. Which of these it contributes to depends on the sign of
$$\langle \alpha_{i'}, u_l \omega_i \rangle = \langle u_l^{-1}\alpha_{i'}, \omega_i \rangle.$$
This sign changes exactly when $u_l^{-1}\alpha_{i'}$ goes from being a positive to a negative root. Thus for $l \leq t_k$, $\omega_k$ contributes to at most the weights $-w_0 \lambda_i$, while for $l \geq t_k$, $\omega_k$ contributes at most to the weights $\mu_i$.

We are now close to proving the proposition. Note that because $u_l = u_{l-1} s_j$, and $s_j \omega_i = \omega_i$, we have that for $i \neq j$,  $u_l \omega_i = u_{l-1} \omega_i$. Note that this is consistent with the fact that 
$$\Delta_{u_{l-1}\omega_{i}, \omega_{i}} = \Delta_{u_{l}\omega_{i}, \omega_{i}}.$$
These are just different names for the function which is associated to the vertex $i$.

Let us first assume that $l \neq t_k$ for any $k$. We need to show that 
$$\sum_{i \neq j} c_{ij} (\lambda_i, \mu_i, \nu_i) + (\lambda_{j_-}, \mu_{j_-}, \nu_{j_-}) + (\lambda_{j_+}, \mu_{j_+}, \nu_{j_+}) = (0,0,\alpha_j).$$
We will need the following identity:
$$s_j \omega_j = -\omega_j - \sum_{i \neq j} c_{ji} \omega_i.$$
This is easy to prove, as $(1+s_j) \omega_j$ is orthogonal to $\alpha_j$, while for $i \neq j$,
$$\langle \alpha_i, (1+s_j) \omega_j \rangle = \langle \alpha_i, s_j \omega_j \rangle = \langle s_j \alpha_i, \omega_j \rangle = -c_{ji}$$
because $s_j \alpha_i = \alpha_i - c_{ji} \alpha_j$. 

Now, applying $u_{l-1}$ to both sides gives
\begin{equation}\label{gradingidentity} u_l \omega_j + u_{l-1} \omega_j = \sum_{i \neq j} - c_{ji} u_{l-1} \omega_i.
\end{equation}
Taking the positive and negative parts of both sides is the same as taking the positive and negative parts term by term gives that 
$$\sum_{i \neq j} c_{ij} \lambda_i + \lambda_{j_-} + \lambda_{j_+} = 0,$$
$$\sum_{i \neq j} c_{ij} \mu_i + \mu_{j_-} + \mu_{j_+} = 0.$$
Finally, we have that
$$\sum_{i \neq j} c_{ij} \nu_i + \nu_{j_-} + \nu_{j_+} = \sum_{i \neq j} c_{ij} \omega_i +2 \omega_j = \alpha_j.$$

Let us now assume that $l = t_k$ for some $k$. Note that the only root that goes from being positive to negative when we apply $s_j$ is $\alpha_j$. Thus, we must have that $u_{l-1}^{-1} \alpha_k = \alpha_j$ and $u_{l}^{-1} \alpha_k = -\alpha_j$. Thus we have that the coefficient of $\omega_k$ in $u_{l-1} \omega_j$ is $1$ and the coefficient in $u_l \omega_j$ is $-1$. The coefficients of the other fundamental weights do not change sign. (In fact, we can say slightly more: $\omega_k$ has a non-zero coefficient in $u_{l-1}\omega_i$ or $u_l \omega_i$ if and only if $i=j$.)

We still have that Equation \ref{gradingidentity} holds. We can proceed similarly as before, except that if we take the positive part of both sides, we get
$$(u_l \omega_j)^+ -\omega_k + (u_{l-1} \omega_j)^+ = (u_l \omega_j + u_{l-1} \omega_j)^+ = \sum_{i \neq j} - c_{ji} (u_{l-1} \omega_i)^+,$$
$$(u_l \omega_j)^-  + (u_{l-1} \omega_j)^- +\omega_k = (u_l \omega_j + u_{l-1} \omega_j)^- = \sum_{i \neq j} - c_{ji} (u_{l-1} \omega_i)^-.$$
Then rearranging gives 
$$\sum_{i \neq j} c_{ij} \lambda_i + \lambda_{j_-} + \lambda_{j_+} = \omega_k^*,$$
$$\sum_{i \neq j} c_{ij} \mu_i + \mu_{j_-} + \mu_{j_+} = \omega_k,$$
as desired.

Proposition \ref{recursive} gives us a way of calculating the weights for the functions in the cluster algebra recursively: it tells us that the weight of the function assigned to the vertex $j_+$ is determined by the weights assigned to the vertices $i$, $j_-$ and possibly $k^\circ$. To put it another way, the structure of the quiver--given by amalgamation of pieces corresponding to letters in a reduced word for $w_0$--determines the weights of the functions. Starting with all the weights of the edge functions, one can use Proposition \ref{recursive} to determine the weights of the face functions. In fact, the relations in Proposition \ref{recursive} overdetermine the weights of the functions, because they also determine the weights on the edge $A_2A_3$, which have already been prescribed.

\section{Weight identities}

We expect that the space $\Conf_3 \A_G$ has as its corresponding $\X$-variety the space $\Conf_3 \B_G$. Thus, for all the unfrozen vertices (which correspond to face functions), we need that the corresponding $X$-coordinate is a function on $\Conf_3 \B_G$. A rational function on $\Conf_3 \A_G$ descends to $\Conf_3 \B_G$ if and only if it is in the $(0,0,0)$-th graded piece of $\mathcal{O}(\Conf_3 \A_G)$. Because we have that
$$p^*(X_i) = \prod_{j \in I}A_j^{b_{ij}},$$
if $A_j$ lies in the graded piece $(\lambda_j, \mu_j, \nu_j)$, we need the following to hold:

\begin{prop} \label{face} If $i \in I$ is the index of a non-frozen vertex, then we have that
$$\sum_{j \in I} b_{ij} (\lambda_j, \mu_j, \nu_j)=(0,0,0)$$
If $i \in I$ is the index of a frozen vertex corresponding to a function of weight $(\omega_k, \omega_k^*,0)$, then 
$$\sum_{j \in I} b_{ij} (\lambda_j, \mu_j, \nu_j)=\frac{1}{2}(\alpha_k,-\alpha_k^*).$$
Similar equations hold for the cyclic shifts of these edge variables.
\end{prop}

Let us verify the first identity. Let us fix a non-frozen vertex in the quiver. Suppose it occurs on level $j$. Then this vertex is involved in two pieces of the quiver coming from the simple reflection $s_j$. Let us suppose that $s_{i_{l_1}}$ and $s_{i_{l_2}}$ are two occurences of $s_j$ in the reduced word for $w_0$ that have no occurence of $s_j$  between them. Let us label the vertices in the piece of the quiver corresponding to $s_{i_{l_1}}$ by $i^1$ (for $i \neq j$), $j_-^1$, $j_+^1$, and possibly $k^{1\circ}$, and label the vertices in the piece of the quiver corresponding to $s_{i_{l_2}}$ by $i^2$ (for $i \neq j$), $j_-^2$, $j_+^2$, and possibly $k^{2\circ}$. Moreover, suppose our vertex corresponds to $j_+^1 = j_-^2$.

Then applying Proposition \label{recursive} gives us that
$$\sum_{i \neq j} c_{ij} (\lambda_{i^1}, \mu_{i^1}, \nu_{i^1}) + (\lambda_{j_-^1}, \mu_{j_-^1}, \nu_{j_-^1}) + (\lambda_{j_+^1}, \mu_{j_+^1}, \nu_{j_+^1}) - (\lambda_{k^{1\circ}}, \mu_{k^{1\circ}}, \nu_{k^{1\circ}})= (0,0,\alpha_j),$$
$$\sum_{i \neq j} c_{ij} (\lambda_{i^2}, \mu_{i^2}, \nu_{i^2}) + (\lambda_{j_-^2}, \mu_{j_-^2}, \nu_{j_-^2}) + (\lambda_{j_+^2}, \mu_{j_+^2}, \nu_{j_+^2}) - (\lambda_{k^{2\circ}}, \mu_{k^{2\circ}}, \nu_{k^{2\circ}})= (0,0,\alpha_j).$$
Subtracting the first equation from the second gives exactly the identity we seek in Proposition (\ref{face}).

Finally, we would like to calculate the expression in the Proposition (\ref{face}) for a frozen variable. It is easy to check that the expression in Proposition (\ref{face}) is invariant under mutation. Thus we may choose to calculate it in a convenient cluster. Let us first treat the case of the function for the edge $A_1A_2$ which lies in the invariant space $[V_{\omega_k^{*}} \otimes V_{\omega_k} \otimes \mathbbm{1}]^G$. Let us consider a cluster coming from a reduced word $w_0=s_{i_1} s_{i_2} s_{i_3} \cdots s_{i_{K-1}} s_{i_K}$ where $s_{i_1}=s_k$. Then the portion of the quiver corresponding to the letter $s_{i_1}$ has vertices that we will label $i$ (for $i \neq k$), $k_-, k_+$ and $k^{\circ}$. As before, let the vertex $?$ be attached to a function of weight $(\lambda_?, \mu_?, \nu_?)$. Then we have

$$(\lambda_{i}, \mu_{i}, \nu_{i}) = (\omega_i^*, 0, \omega_i),$$
$$(\lambda_{k_-}, \mu_{k_-}, \nu_{k_-}) = (\omega_k^*, 0, \omega_k),$$
$$(\lambda_{k_+}, \mu_{k_+}, \nu_{k_+}) = (\sum_{i \neq k}-c_{ki} \omega_i^*, \omega_k, \omega_k),$$
$$(\lambda_{k^\circ}, \mu_{k^\circ}, \nu_{k^\circ}) = (\omega_k^*, \omega_k, 0).$$

To calculate the weight $(\lambda_{k_+}, \mu_{k_+}, \nu_{k_+})$ we use that $s_k \omega_k = \omega_k - \alpha_k = \sum_{i \neq k} -c_{ki} \omega_i - \omega_k$.

Now we have that there is an arrow from $k_-$ to $k^\circ$, and from $k^\circ$ to $k_+$. Moreover, there are half-arrows from the vertex $i^\circ$ to $k^\circ$ whenever $i$ and $k$ are adjacent in the Dynkin diagram. Thus, for the vertex $k^\circ$, the expression from Proposition (\ref{face}) becomes
$$(\omega_k^*, 0, \omega_k) - (\sum_{i \neq k}-c_{ki} \omega_i^*, \omega_k, \omega_k) + \frac{1}{2} \sum_{i \neq k} -c_{ki} (\omega_i^*, \omega_i, 0)$$
which is equal to
$$\frac{1}{2} (2\omega_k^* + \sum_{i \neq k} c_{ki} \omega_i^*, - 2\omega_k - \sum_{i \neq k} c_{ki} \omega_i, 0) = \frac{1}{2} (\alpha_{k}^*, -\alpha_k, 0).$$
This is precisely the expectation from \cite{Le2}.

Let us do the same computation for the vertex $k_-$. There is an arrow from $k_+$ to $k_-$ and from $k_-$ to $k^\circ$. Moreover there are half-arrows from $k_-$ to $i$ whenever $i$ and $k$ ar adjacent in the Dynkin diagram. The expression becomes
$$(\sum_{i \neq k}-c_{ki} \omega_i^*, \omega_k, \omega_k) - (\omega_k^*, \omega_k, 0) - \frac{1}{2} \sum_{i \neq k} -c_{ki} (\omega_i^*, 0, \omega_i) = $$
$$\frac{1}{2}(-2\omega_k^* - \sum_{i \neq k} c_{ki} \omega_i^*, 0 , 2\omega_k + \sum_{i \neq k} c_{ki} \omega_i) = (-\alpha_{k}^*, 0, \alpha_k).$$
This handles the case of functions for the edge $A_1A_3$. A similar computation can be performed from the edge $A_2A_3$ using a cluster where the last letter $s_{i_K}$ in the reduced word for $w_0$ is chosen to be $s_k$.

If we amalgamate the cluster structure on triangles to build up cluster structures for other surfaces, we will amalgamate along edges and unfreeze the edge variables. The formula in Proposition (\ref{face}) for frozen edge variables guarantee that after amalgamation, we will have $\sum_{j \in I} b_{ij} (\lambda_j, \mu_j, \nu_j)=(0,0,0)$ for $i$ the index of a frozen variable that becomes unfrozen.

\section{$\Conf_4 \A_G$ and the largest double Bruhat cell}

We previously explained how the cluster structure on $\Conf_4 \A_G$ can be related to the one on $B^-$ (or if one prefers, the one on $B^+$). In this section, we explain how the cluster structure on $\Conf_4 \A_G$ similarly comes from the cluster structure on the double Bruhat cell $G^{w_0,w_0}$. We let $(A_1,A_2,A_3,A_4) \in  \Conf_4 \A_G$.

First, recall the double Bruhat cell
$$G^{u,v}:= B^+uB^+ \cap B^-vB^-.$$

The constructions in \cite{BFZ} give a cluster for $G^{u,v}$ for every reduced word for $(u,v) \in W \times W$. A reduced word for $(u,v)$ is a shuffle of a reduced word for $u$ and a reduced word for $v$. The reduced word for $u$ will be in the letters $-1, -2, \dots, -n$, while the word for $v$ will be in the letters $1, 2, \dots, n$. The reflections $s_{-i}$ generate the first copy of $W$, while the reflections $s_i$ generate the second copy of $W$. For example, for $SL_4$, we could take $$(w_0,w_0)=s_{-1}s_{-2}s_{-3}s_{-1}s_{-2}s_{-1}s_{3}s_{2}s_{1}s_{3}s_{2}s_{3}.$$ Note that we use negative indices for $u$, whereas in previous sections of this paper we used regular indices for $u$.

Consider the natural map from $G^{w_0,w_0}$ to $\Conf_4 \A_G$ given by the formula
$$i: g \in G^{w_0,w_0} \rightarrow (U^-, \overline{w_0}U^-, g \cdot \overline{w_0}U^-, g \cdot U^-) \in \Conf_4 \A_{G}.$$
This is an injective map. Functions on $\Conf_4 \A_G$ can be pulled back to give functions on $G^{w_0,w_0}$. The cluster variables on $G^{w_0,w_0}$pull back to give a subset of the cluster variables on $\Conf_4 \A_G$. More precisely, the edge variables for the edges $A_1A_2$ and $A_3A_4$ take the value $1$ on the image $i(G^{w_0,w_0})$, while the remaining cluster variables pull back to give cluster variables on $G^{w_0,w_0}$.

Let $s_{i_1} \cdots s_{i_{2K}}$ be a reduced word for $(w_0,w_0)$. To construct the quiver for the corresponding cluster for $G^{w_0,w_0}$, we amalgamate pieces corresponding to letters $s_i$ or $s_{-i}$. 

We can similarly associate a cluster for $\Conf_4 \A_G$ to every reduced word for $(w_0,w_0) \in W \times W$. To obtain this quiver, we also use the amalgamation procedure, but some pieces will be modified. Define $(u_l,v_l) = s_{i_1} \cdots s_{i_{l}}$. Let $t_k$ be the index such that $u_i^{-1} \alpha_k$ first becomes a negative root when $i=t_k$. Similarly, let $r_k$ be the index such that $v_i^{-1} \alpha_k$ first becomes a negative root when $i=r_k$. Then for the letters of the reduced word which are $t_k$ or $r_k$, we use modified pieces. 

Let us describe the procedure in more detail. The pieces corresponding to $s_{-i}$ will be the pieces as depicted in Figure 2, while the modified pieces corresponding to $s_{-i}$ will be as depicted in Figure 6. (These were the pieces we previously labelled $s_i$.) The modified pieces will be used for $i=t_k$. For these pieces, we will have extra vertices $k^\circ$ whose functions belong to the invariant space $[V_{\omega_k^{*}} \otimes V_{\omega_k} \otimes \mathbbm{1} \otimes \mathbbm{1}]^G$. These will be functions for the edge $A_1A_2$.

Let us now describe the modified pieces associated to $s_i$. Essentially, they look the same as those for $s_{-i}$, but with the arrows reversed. Suppose that $s_{i_{r_k}} = s_j$. Then the quiver for this simple reflection will consist of $n+2$ vertices. There will be two vertices $j_-$ and $j_+$ on the $j$-th row as before. There will be one vertex $i$ on the $i$-th row for each every other node $i \neq j$ of the Dynkin diagram. Moreover, the quiver will have a vertex $k^{\bullet}$, which will correspond to the edge function for the edge $A_3A_4$ in the invariant space $[\mathbbm{1} \otimes \mathbbm{1} \otimes V_{\omega_k^{*}} \otimes V_{\omega_k} ]^G$

The quiver will have the following arrows:

\begin{itemize}
\item An arrow from $j_-$ to $j_+$.
\item A dotted arrow from $j_+$ to $i$ whenever $i$ and $j$ are adjacent in the Dynkin diagram.
\item A dotted arrow from $i$ to $j_-$ whenever $i$ and $j$ are adjacent in the Dynkin diagram.
\item An arrow from $j_+$ to $k^{\bullet}$.
\item An arrow from $k^{\bullet}$ to $j_-$.
\end{itemize}

Here is how the pieces $s_1, s_2, s_3$ would look if they occured as $s_{i_{r_k}}$:

\begin{center}
\begin{tikzpicture}[scale=2]

  \node (x211) at (-2,0.9) {\Large $\ontop{2}{\bullet}$};

  \node (x121) at (-2.5,0) {\Large $\ontop{1_-}{\bullet}$};
  \node (x112) at (-1.5,0) {\Large $\ontop{1_+}{\bullet}$};
  \node (12k) at (-2,1.5) {\Large $\ontop{k^{\bullet}}{\bullet}$};

  \draw [->, postaction={decorate}] (x121) -- (x112) node [midway, above] {$s_1$};
  \draw [->, dashed] (x211) to (x121);
  \draw [->, dashed] (x112) to (x211);
  \draw [->] (x112) to (12k);
  \draw [->] (12k) to (x121);

  \node (x310) at (0,0.9) {\Large $\ontop{3}{\bullet}$};

  \node (x220) at (-0.5,0) {\Large $\ontop{2_-}{\bullet}$};
  \node (x211) at (0.5,0) {\Large $\ontop{2_+}{\bullet}$};

  \node (x121) at (0,-0.9) {\Large $\ontop{1}{\bullet}$};
  \node (12k) at (0,1.5) {\Large $\ontop{k^{\bullet}}{\bullet}$};

  \draw [->, postaction={decorate}] (x220) -- (x211) node [midway, above] {$s_2$};

  \draw [->, dashed] (x211) to (x121);
  \draw [->, dashed] (x310) to (x220);

  \draw [->, dashed] (x121) to (x220);
  \draw [->, dashed] (x211) to (x310);

  \draw [->] (12k) to (x220);
  \draw [->] (x211) to (12k);

  \node (x310) at (1.5,0) {\Large $\ontop{3_-}{\bullet}$};
  \node (x301) at (2.5,0) {\Large $\ontop{3_+}{\bullet}$};

  \node (x211) at (2,-0.9) {\Large $\ontop{2}{\bullet}$};
  \node (12k) at (2,1.5) {\Large $\ontop{k^{\bullet}}{\bullet}$};

  \draw [->, postaction={decorate}] (x310) -- (x301) node [midway, above] {$s_3$};
  \draw [->,dashed] (x211) to (x310);
  \draw [->,dashed] (x301) to (x211);
  
  \draw [->] (12k) to (x310);
  \draw [->] (x301) to (12k);

\draw[yshift=-1.2cm]
  node[below,text width=6cm] % ,style=information text --- don't know this one
  {
  Figure 5. Modified pieces of the quiver for $s_1, s_2, s_3$ when a simple root changes sign for $\Conf_3 \A_G$.
  };

\end{tikzpicture}
\end{center}

The unmodified pieces can be obtained by removing the vertices $k^\bullet$. To build the quiver for $\Conf_4 \A_G$ corresponding to a reduced word for $(w_0,w_0)$, we amalgamate the pieces corresponding to the simple reflections in the reduced word decomposition for $(w_0,w_0)$, using the modified pieces for the $s_{i_{t_k}}$ and $s_{i_{r_k}}$. The last step is that we need to add some arrows between the vertices $k^{\circ}$ and between the vertices $k^\bullet$. We add a dotted arrow from $k^{\circ}$ to $k'^{\circ}$ whenever nodes $k$ and $k'$ are adjacent in the Dynkin diagram for $G$, and $t_{k'} < t_k$. We also add a dotted arrow from $k^{\bullet}$ to $k'^{\bullet}$ whenever nodes $k$ and $k'$ are adjacent in the Dynkin diagram for $G$, and $r_k < r_{k'}$. This completes the construction of the quiver for $\Conf_4 \A_G$.

We have already described the functions attached to $k^\circ$ and $k^\bullet$. We now describe the remaining functions.

Let us consider the $l$-th piece of the quiver corresponding to the letter $s_{i_l}$. Recall that $(u_l,v_l) = s_{i_1} \cdots s_{i_{l}}$. Then set
$$-w_0 \lambda_i = (u_l \omega_i)^+,$$
$$\mu_i = -(u_l \omega_i)^-,$$
$$w_0 \nu_i = (v_l \omega_i^*)^-,$$
$$\kappa_i = (v_l \omega_i^*)^+.$$
Thus $\lambda, \mu, \nu, \kappa$ are chosen minimally so that
$$u\omega_i = -w_0 \lambda - \mu,$$
$$v\omega_i^* = w_0 \nu + \kappa.$$
Then the functions attached to the vertex $i \neq j$ or $j_+$ lie in the invariant spaces 
$$[V_{\lambda_i} \otimes V_{\mu_i} \otimes V_{\nu_i} \otimes V_{\kappa_i} ]^G,$$
$$[V_{\lambda_j} \otimes V_{\mu_j} \otimes V_{\nu_j} \otimes V_{\kappa_j} ]^G,$$
respectively. These functions pull back to $G^{w_0,w_0}$ to give the reduced minors $\Delta_{u \omega_{i}, v\omega_{i}}$.

\end{document}